\documentclass[11pt]{amsart}
\usepackage{amsmath,amsfonts,geometry,amsthm,enumerate}
\theoremstyle{definition}
\newtheorem{defi}{Definition}
\theoremstyle{plain}
\newtheorem*{theo}{Theorem}
\begin{document}
\title[Symplectic connections]{Symplectic Connections on Supermanifolds:  Existence and Non-Uniqueness}
\author{Paul A. Blaga}
\address{``Babe\c{s}-Bolyai'' University, \\ Faculty of Mathematics and Computer Sciences\\
1, Kog\u{a}lniceanu Street,\\
400084 Cluj-Napoca,\\
Romania}
\email{pablaga@cs.ubbcluj.ro}
\subjclass{58A50, 53D05}
\keywords{symplectic supermanifolds,symplectic connections}
\begin{abstract}
We show, in this note, that on any symplectic supermanifold, even or odd, there exist an infinite
dimensional affine space of symmetric connections, compatible to the symplectic form.
\end{abstract}
\dedicatory{To the memory of Professor Mircea-Eugen Craioveanu (1942-2012)}
\maketitle
\section{Introduction}
The connections compatible with a symplectic form have been studied for several decades, by now. They were introduced by Ph. Tondeur, in 1961 (see~\cite{tondeur}), for the more general situation of an almost-symplectic manifold.
Nevertheless, they became really important lately, in the early ninetieth, when Fedosov (\cite{fedosov}) discovered
that they may be useful in the deformation quantization. Therefore, a symplectic manifold endowed
with a symmetric connection, compatible with the symplectic form, has been baptized with the name
of \emph{Fedosov manifold}. A recent review of the theory of symplectic connections can be found in \cite{bieliavsky}. A few years later, the notion of symplectic connection has been
extended to symplectic supermanifolds and the corresponding objects (namely symplectic
supermanifolds, even or odd, endowed with a symplectic connection) have been named \emph{Fedosov
supermanifolds} (see \cite{geyer}). It is the aim of this note to show that, as in the case of symplectic manifolds,
on a symplectic supermanifold (odd or even, it doesn't matter), symplectic connections exist in
abundance. The language we use is slightly different from that used in the original papers, because
we use a coordinate-free approach (see \cite{bejancu90}, \cite{bejancu91}, \cite{blaga}).

As it is well-known, there are several approaches to supermanifolds, not entirely equivalent. The
differences are not very important for this paper. Nevertheless, to avoid ambiguities, we state from
the very beginning that for us ``supermanifold'' means ``supermanifold in the sense of Berezin and
Leites''\footnote{These supermanifolds are also called ``graded manifolds'', especially in the
Western literature.}. For details, see \cite{bartocci}, \cite{constantinescu}, \cite{kostant}, \cite{leites}.
\section{Symplectic connections on supermanifolds}
\begin{defi}
Let $\mathcal{M}$ be an arbitrary, finite dimensional, supermanifold. A
connection (a covariant derivative) on this supermanifold is a mapping
$\nabla:\mathcal{X}(\mathcal{M})\times \mathcal{X}(\mathcal{M})\to
\mathcal{X}(\mathcal{M})$ for which the following conditions are fulfilled:
\begin{enumerate}[(i)]
\item $\nabla$ is additive in both arguments:
\begin{equation*}
\nabla_{X_1+Y_2}{Y}=\nabla_{X_1}Y+\nabla_{X_2}Y,\quad
\nabla_X(Y_1+Y_2)=\nabla_XY_1+\nabla_XY_2;
\end{equation*}
\item $\nabla_{fX}Y=f\nabla_XY$;
\item $\nabla_X(fY)=X(f)\cdot Y+(-1)^{|X|\cdot |f|}\nabla_XY$,
\end{enumerate}
where in the first two relations $X,Y, X_1, X_2, Y_1, Y_2$ are arbitrary vector
fields and $f$ an arbitrary superfunction, while in the last equality all the
entries are assumed to be homogeneous.
\end{defi}
The torsion tensor can be defined here in a similar manner to the corresponding
tensor for connections on ordinary (ungraded) manifolds:
\begin{defi}
Let $\nabla$ be a connection on a supermanifold. The \emph{torsion} of the
connection is the tensor field (twice covariant and once contravariant) defined
by
\begin{equation*}
T(X,Y)=\nabla_XY-(-1)^{|X|\cdot |Y|}\nabla_YX-[X,Y],
\end{equation*}
for any homogeneous vector fields $X$ and $Y$. Also by analogy with the
classical case, a connection on a supermanifold is called \emph{symmetric} if
its torsion vanishes. Thus, the connection is symmetric iff for any homogeneous
vector fields $X$ and $Y$ we have
\begin{equation*}
\nabla_XY-(-1)^{|X|\cdot |Y|}\nabla_YX=[X,Y].
\end{equation*}
\end{defi}
It can be shown easily that, using the same methods from the classical
differential geometry, the covariant derivative on supermanifolds can be
extended to arbitrary tensor fields, not just vector fields. The interesting
case for us is the one of twice covariant tensor fields. Thus, if $g$ is a
twice covariant \emph{homogeneous} tensor field on a supermanifold
$\mathcal{M}$, then we have
\begin{equation*}
\begin{split}
(\nabla_Xg)(Y,Z)&\equiv \nabla_Xg(Y,Z)=X(g(Y,Z))-(-1)^{|X|\cdot
|g|}g(\nabla_XY,Z)-\\
&-(-1)^{|X|\cdot(|Y|+|g|)}g(Y,\nabla_XZ).
\end{split}
\end{equation*}
We are interested, in this paper, in the particular case of a
\emph{homogeneous} symplectic supermanifold, i.e. a supermanifold endowed with
a homogeneous 2-form $\omega$, which is both closed and non-degenerate.
\begin{defi}
Let $(\mathcal{M},\omega)$ be a homogeneous symplectic supermanifold
(hereafter, it will be called, simply, \emph{symplectic supermanifold}). A
connection $\nabla$ on $\mathcal{M}$ is called \emph{symplectic} it is both
symmetric and compatible to the symplectic form. Thus, a symplectic connection
on a symplectic supermanifold is a connection $\nabla$ for which:
\begin{enumerate}[(i)]
\item the torsion tensor vanishes, i.e.
\begin{equation*}
\nabla_XY-(-1)^{|X|\cdot |Y|}\nabla_YX=[X,Y]
\end{equation*}
and
\item it is compatible to the symplectic form, i.e.
\begin{equation*}
\begin{split}
 \nabla_X\omega(Y,Z)&=X(\omega(Y,Z))-(-1)^{|X|\cdot
|\omega|}\omega(\nabla_XY,Z)-\\
&-
(-1)^{|X|\cdot(|Y|+|\omega|)}\omega(Y,\nabla_XZ)=0,
\end{split}
\end{equation*}
\end{enumerate}
for any homogeneous vector fields $X,Y,Z$.
\end{defi}
\section{Existence and uniqueness results for symplectic connections}
\begin{theo}[Existence]
Let $(\mathcal{M},\omega)$ be a symplectic supermanifold. Then on $\mathcal{M}$
there is at least a  symplectic connection.
\end{theo}
\begin{proof}
The proof we are going to give is an adaptation of the proof from the classical
symplectic geometry of manifolds. Namely, we notice, first of all, that on $M$
there is at least a symmetric connection, $\nabla^0$. To proof this, it is
enough to consider a Riemannian metric on $\mathcal{M}$ (which we know we can
find) and take $\nabla^0$ to be the Levi-Civita connection associated to this
metric, which, we also know, exists (and it is even unique). Of course,
$\nabla^0$ is not a symplectic connection, in most situations, and what we
shall do is to ``correct'' this connection to get a symplectic one.

We define now a twice covariant and once contravariant tensor field $N$ through
the relation
\begin{equation}\label{proof0}
\nabla^0_X\omega(Y,Z)=(-1)^{|\omega|\cdot |X|}\omega(N(X,Y),Z).
\end{equation}
We shall proof some properties of $N$, for later use. First, we claim that
\begin{equation}\label{proof1}
\omega(N(X,Y),Z)=-(-1)^{|Y|\cdot |Z|}\omega (N(X,Z),Y).
\end{equation}
Indeed, we have
\begin{equation*}
\begin{split}
&\omega(N(X,Y),Z)=(-1)^{|\omega|\cdot
|X|}\nabla_X^0\omega(Y,Z)=\\
&=-(-1)^{|\omega|\cdot |X|}(-1)^{|Y|\cdot
|Z|}\nabla^0_X(Z,Y)=-(-1)^{|Y|\cdot |Z|}\omega (N(X,Z),Y).
\end{split}
\end{equation*}
Another important property of $N$, which follows, this time, from the closeness
of the symplectic form, is the following:
\begin{equation}\label{proof3}
\begin{split}
&\omega(N(X,Y),Z)+(-1)^{|X|(|Y|+|Z|)}\omega(N(Y,Z),X)+\\
&+(-1)^{|Z|(|X|+|Y|)}\omega(N(Z,X),Y)=0�
\end{split}
\end{equation}
As mentioned before, to prove~(\ref{proof3}), we shall start from the closeness
of the symplectic form and we shall use the symmetry of the connection
$\nabla^0$, as well as the definition of the tensor $N$. Thus, we have
\begin{equation*}
\begin{split}
0&=d\omega(X,Y,Z)=(-1)^{|\omega|\cdot
|X|}X(\omega(Y,Z))-\\
&-(-1)^{|Y|(|\omega|+|X|)}Y(\omega(X,Z))+(-1)^{|Z|(|\omega|+|X|+|Y|)}Z(\omega(X,Y))-\\
&-\omega([X,Y],Z)+(-1)^{|Y|\cdot
|Z|}\omega([X,Z],Y)-(-1)^{|X|(|Y|+|Z|)}\omega([Y,Z],X)=\\
&=
(-1)^{|\omega|\cdot
|X|}X(\omega(Y,Z))-(-1)^{|Y|(|\omega|+|X|)}Y(\omega(X,Z))+\\
&+(-1)^{|Z|(|\omega|+|X|+|Y|)}Z(\omega(X,Y))-\omega\left(\nabla^0_XY-(-1)^{|X|\cdot
|Y|}\nabla^0_YX,Z\right)+\\
&+(-1)^{|Y|\cdot |Z|}\omega\left(\nabla^0_XZ-(-1)^{|X|\cdot
|Z|}\nabla_Z^0X,Y\right)-\\
&- (-1)^{|X|(|Y|+|Z|)}\omega\left(\nabla^0_YZ-(-1)^{|Y|\cdot
|Z|}\nabla^0_ZY,X\right)=\\
&=
(-1)^{|\omega|\cdot |X|}X(\omega(Y,Z))-(-1)^{|Y|(|\omega|+|X|)}Y(\omega(X,Z))+\\
&+
(-1)^{|Z|(|\omega|+|X|+|Y|)}Z(\omega(X,Y))-
\omega\left(\nabla^0_XY,Z\right)+\\
&+
(-1)^{|X|\cdot
|Y|}\omega\left(\nabla^0_YX,Z\right)+ (-1)^{|Y|\cdot|Z|}\omega\left(\nabla^0_XZ,Y\right)
-\\
&-
(-1)^{(|X|+|Y|)|Z|}\omega\left(\nabla_Z^0X,Y\right)-(-1)^{|X|(|Y|+|Z|)}\omega\left(\nabla^0_YZ,X)\right)+\\
&
+(-1)^{|X|(|Y|+|Z|)+|Y|\cdot|Z|}
\omega\left(\nabla^0_ZY,X\right)=(-1)^{|\omega|\cdot
|X|}\left[X(\omega(Y,Z))-\right.\\
&\left.-
(-1)^{|\omega|\cdot|X|}\omega\left(\nabla^0_XY,Z\right)-
(-1)^{|X|(|\omega|+|Y|}\omega\left(Y,\nabla^0_XZ\right)\right]-\\
&-(-1)^{|Y|(|\omega|+|X|}\Big[Y(\omega(X,Z))-(-1)^{|\omega|\cdot
|Y|}\omega\left(\nabla^0_YX,Z\right)-\\
&-
(-1)^{|Y|(|\omega|+|X|}\omega\left(X,\nabla^0_YZ\right)\Big]+
(-1)^{|Z|(|\omega|+|X|+|Y|}\Big[Z(\omega(X,Y))-\\
&-(-1)^{|\omega|\cdot|Z|}\omega\left(\nabla^0_ZX,Y\right)-
(-1)^{|Z|(|\omega|+|X|}
\omega\left(X,\nabla^0_ZY\right)\Big]=\\
&=
(-1)^{|\omega|\cdot|X|}\nabla^0_X
\omega(Y,Z)-(-1)^{|Y|(|\omega|+|X|}\nabla^0_Y\omega(X,Z)+\\
&+(-1)^{|Z|(|\omega|+|X|+|Y|)}\nabla^0_Z\omega(X,Y)
\end{split}
\end{equation*}
We define now a new connection, $\nabla$, by letting
\begin{equation}\label{sympl1}
 \nabla_XY=\nabla^0_XY+\frac{1}{3}N(X,Y)+\frac{(-1)^{|X|\cdot|Y|}}{3}N(Y,X).
\end{equation}
We start by proving that this is, indeed, a connection. $\nabla$
is, obviously,  bi-additive and homogeneous in the first variable.
Moreover, we have
\begin{equation*}
\begin{split}
 \nabla_X(fY)&=\nabla^0_X(fY)+\frac{1}{3}N(X,fY)+\frac{(-1)^{|X|\cdot|Y|}}{3}N(fY,X)=\\
&=f\nabla^0_XY+(-1)^{|f|\cdot|X|}X(f)\cdot Y+f\bigg(\frac{1}{3}N(X,Y)+\\
&+
\frac{(-1)^{|X|\cdot|Y|}}{3}N(Y,X)\bigg)=f\nabla_XY+(-1)^{|f|\cdot|X|}X(f)\cdot Y,
\end{split}
\end{equation*}
hence $\nabla$ is a connection.

We claim that $\nabla$ is a symplectic connection. Let's check first that $\nabla$ is symmetric. Indeed, we have
\begin{equation*}
 \begin{split}
  \nabla_XY&-(-1)^{|X|\cdot|Y|}\nabla_YX=\nabla^0_XY+\frac{1}{3}N(X,Y)+\frac{(-1)^{|X|\cdot|Y|}}{3}N(Y,X)-\\
&-(-1)^{|X|\cdot|Y|}\left(\nabla^0_YX+\frac{1}{3}N(Y,X)+\frac{(-1)^{|Y|\cdot|X|}}{3}N(X,Y)\right)=\\
&=\nabla^0_XY-(-1)^{|X|\cdot|Y|}\nabla^0_YX=[X,Y],
 \end{split}
\end{equation*}
where we used the fact that the connection $\nabla^0$ is
symmetric. Finally, we show that  the connection is compatible
with the symplectic form. We have
\begin{equation*}
\begin{split}
& \nabla_X\omega(Y,Z)=X(\omega(Y,Z))-(-1)^{|\omega|\cdot|X|}\omega\left(\nabla_XY,Z\right)-\\
&-(-1)^{|X|(|\omega|+|Y|}\omega\left(Y,\nabla_XZ\right)=X(\omega(Y,Z))-(-1)^{|\omega|\cdot|X|}\omega\Big(\nabla^0_XY+\\
&+\frac{1}{3}N(X,Y)+
 \frac{(-1)^{|X|\cdot|Y|}}{3}N(Y,X),Z\Big)-(-1)^{|X|(|\omega|+|Y|}\omega\Big(Y,\nabla^0_XZ+\\
 &+
 \frac{1}{3}N(X,Z)+\frac{(-1)^{|X|\cdot|Z|}}{3}N(Z,X)\Big)=X(\omega(Y,Z))-\\
 &-(-1)^{|\omega|\cdot|X|}\omega\left(\nabla^0_XY,Z\right)-(-1)^{|X|(|\omega|+|Y|)}
\omega\left(Y,\nabla^0_XZ\right)-\\
&-\frac{1}{3}(-1)^{|\omega|\cdot|X|}\omega(N(X,Y),Z)-\frac{1}{3}(-1)^{|X|(|\omega|+|Y|)}\omega(N(Y,X),Z)-\\
&-\frac{1}{3}(-1)^{|X|(|\omega|+|Y|)}\omega(Y,N(X,Z))-\frac{1}{3}(-1)^{|X|(|\omega|+|Y|+Z|)}\omega(Y,N(Z,X))=\\
&=\nabla^0_X\omega(Y,Z)-\frac{1}{3}(-1)^{|\omega|\cdot|X|}\omega(N(X,Y),Z)+\\
&+
\frac{1}{3}(-1)^{|X|(|\omega|+|Y|+|Z|)}\omega(N(Y,Z),X)-\frac{1}{3}(-1)^{|\omega|\cdot|X|}\omega(N(X,Y),Z)+\\
&+\frac{1}{3}(-1)^{|\omega|\cdot|X|+|Z|(|X|+|Y|}\omega(N(Z,X),Y)=
(-1)^{|\omega|\cdot|X|}\omega(N(X,Y),Z)-\\
&-
\frac{2}{3}(-1)^{|\omega|\cdot|X|}\omega(N(X,Y),Z)+\\
&+
\frac{1}{3}(-1)^{|\omega|\cdot|X|}\Big((-1)^{|X|(|Y|+|Z|)}\omega(N(Y,Z),X)+\\
&+
(-1)^{|Z|(|X|+|Y|)}\omega(N(Z,X),Y)\Big)=\frac{1}{3}(-1)^{|\omega|\cdot|X|}\Big(\omega(N(X,Y),Z)+\\
&+(-1)^{|X|(|Y|+|Z|)}\omega(N(Y,Z),X)+(-1)^{|Z|(|X|+|Y|)}\omega(N(Z,X),Y) \Big)=0,
 \end{split}
\end{equation*}
which proves that, indeed, $\nabla$ is a symplectic connection.

Thus, on any symplectic supermanifold there is \emph{at least} a symplectic connection.
 As we shall prove next, there are, actually, infinitely many.

We notice, first of all, that the difference of two symplectic connections is allways a
symplectic connection. Let now $\nabla$ be a symplectic connection. Any other connection on
$\mathcal{M}$ should be of the form
\begin{equation*}
 \nabla'_XY=\nabla_XY+ S(X,Y),
\end{equation*}
where $S$ is a $(2,1)$ tensor field on $\mathcal{M}$. If we want $\nabla'$ to be symplectic,
first of all it should be symmetric, which means:
\begin{equation*}
\nabla'_XY-(-1)^{|X|\cdot|Y|}\nabla'_YX = [X,Y],
\end{equation*}
i.e.
\begin{equation*}
 \nabla_XY+S(X,Y)-(-1)^{|X|\cdot|Y|}\nabla_YX-(-1)^{|X|\cdot|Y|}S(Y,X)=[X,Y].
\end{equation*}
As $\nabla$ is symmetric, it follows that $S$ should verify the relation
\begin{equation*}
 S(X,Y)=(-1)^{|X|\cdot|Y|}S(Y,X),
\end{equation*}
meaning that $S$ is supersymmetric. Now we should ask that
$\nabla'$ should, also, be compatible to the symplectic form. We
have:
\begin{equation*}
\begin{split}
&\nabla'_X\omega(Y,Z)=X\left(\omega(Y,Z)\right)-(-1)^{|\omega|\cdot
|X|}\omega\left(\nabla'_XY,Z\right)-\\
&-(-1)^{|X|(|\omega|+|Y|}\omega\left(Y,\nabla'_XZ\right)=\\
&=\underbrace{X\left(\omega(Y,Z)\right)-(-1)^{|\omega|\cdot
|X|}\omega\left(\nabla_XY,Z\right)-(-1)^{|X|(|\omega|+|Y|}\omega\left(Y,\nabla_XZ\right)}_{=0}-\\
&-(-1)^{|\omega|\cdot
|X|}\omega\left(S(X,Y),Z\right)-(-1)^{|X|(|\omega|+|Y|}\omega\left(Y,S(X,Z)\right)=\\
&=(-1)^{|\omega|\cdot |X|}\left[\omega\left(S(X,Y),Z\right)+(-1)^{|X|\cdot
|Y|}\omega\left(Y,S(X,Z)\right)\right]=\\
&=(-1)^{|\omega|\cdot |X|}\left[\omega\left(S(X,Y),Z\right)-(-1)^{|Y|\cdot
|Z|}\omega\left(S(X,Z),Y\right)\right].
\end{split}
\end{equation*}
Thus, $\nabla'$ is a symplectic connection if and only if
\begin{equation*}
\omega\left(S(X,Y),Z\right)=(-1)^{|Y|\cdot |Z|}\omega\left(S(X,Z),Y\right),
\end{equation*}
i.e. the 3-covariant tensor field $\omega\left(S(X,Y),Z\right)$ is \emph{totally  graded
symmetric}. The conclusion is, as in the classical, ungraded, case, that the set of all symplectic
connections on a given symplectic supermanifold is an infinite dimensional affine space.
\end{proof}

\end{document}